\documentclass[11pt]{amsart}
\usepackage[letterpaper,margin=1.2in]{geometry}
\usepackage{amsbsy,amsfonts,amsmath,amssymb,amscd,amsthm,mathrsfs}
\usepackage{subfigure,enumitem,graphicx,epsfig,color}
\usepackage[colorlinks=true,linkcolor=blue,citecolor=green]{hyperref}
\normalsize

\newtheorem{theorem}{Theorem}[section]
\newtheorem{lemma}[theorem]{Lemma}

\newtheorem{definition}[theorem]{Definition}

\newtheorem{remark}[theorem]{Remark}

\newtheorem{proposition}[theorem]{Proposition}

\newtheorem*{theorem A}{Theorem A}
\newtheorem*{corollary B}{Corollary B}
\newtheorem*{theorem C}{Theorem C}
\newtheorem*{corollary D}{Corollary D}
\theoremstyle{definition}

\begin{document}
\title[Variational principle for neutralized Bowen topological entropy of NDSs]{Variational principle for neutralized Bowen topological entropy on subsets of non-autonomous dynamical systems}
\author[J. Nazarian Sarkooh]{Javad Nazarian Sarkooh$^{*}$}
\address{Department of Mathematics, Ferdowsi University of Mashhad, Mashhad, IRAN.}
\email{\textcolor[rgb]{0.00,0.00,0.84}{javad.nazariansarkooh@gmail.com}}
\author[A. Ehsani]{Azam Ehsani}
\address{Department of Mathematics, Ferdowsi University of Mashhad, Mashhad, IRAN.}
\email{\textcolor[rgb]{0.00,0.00,0.84}{aza\_ehsani@yahoo.com}}
\author[Z. Pashaei]{Zeynal Pashaei}
\address{Department of Mathematics, Ferdowsi University of Mashhad, Mashhad, IRAN.}
\email{\textcolor[rgb]{0.00,0.00,0.84}{szpashaei1986@yahoo.com}}
\author[R. Abdi]{Roghayeh Abdi}
\address{Department of Mathematics, Ferdowsi University of Mashhad, Mashhad, IRAN.}
\email{\textcolor[rgb]{0.00,0.00,0.84}{a.roghayeh86@gmail.com}}
 \thanks{$^*$Corresponding author}
\begin{abstract}
Ovadia and Rodriguez-Hertz \cite{SBOFRH} defined the neutralized Bowen open ball for an autonomous dynamical system $(X,f)$ on a compact metric space $(X,d)$ as
\begin{equation*}
B_{n}(x,\text{e}^{-n\epsilon})=\{y\in X: d(f^{j}(x),f^{j}(y))<\text{e}^{-n\epsilon}\ \text{for all}\ 0\leq j\leq n-1\},
\end{equation*}
where $\epsilon>0$, $n\in\mathbb{N}$, and $x\in X$. Replacing the usual Bowen open ball with neutralized Bowen open ball, we introduce the notions of neutralized Bowen topological entropy of subsets, neutralized weighted Bowen topological entropy of subsets, lower neutralized Brin-Katok's local entropy of Borel probability measures, and neutralized Katok's entropy of Borel probability measures for a non-autonomous dynamical system $(X,f_{1,\infty})$ on a compact metric space $(X,d)$. Then, we establish the following variational principles for neutralized Bowen topological entropy and neutralized weighted Bowen topological entropy of non-empty compact subsets in terms of lower neutralized Brin-Katok's local entropy and neutralized Katok's entropy.
\begin{eqnarray*}
h_{top}^{NB}(f_{1,\infty},Z)=h_{top}^{NWB}(f_{1,\infty},Z)
&=& \lim_{\epsilon\to 0}\sup\{\underline{h}_{\mu}^{NBK}(f_{1,\infty},\epsilon):\mu\in\mathcal{M}(X), \mu(Z)=1\}\\
&=& \lim_{\epsilon\to 0}\sup\{h_{\mu}^{NK}(f_{1,\infty},\epsilon):\mu\in\mathcal{M}(X), \mu(Z)=1\},
\end{eqnarray*}
where $\mathcal{M}(X)$ denotes the set of all Borel probability measures on $X$ and $Z$ is a non-empty compact subset of $X$. Moreover, $h_{top}^{NB}(f_{1,\infty},Z)$, $h_{top}^{NWB}(f_{1,\infty},Z)$, $\underline{h}_{\mu}^{NBK}(f_{1,\infty},\epsilon)$, and $h_{\mu}^{NK}(f_{1,\infty},\epsilon)$ are the neutralized Bowen topological entropy of $Z$, neutralized weighted Bowen topological entropy of $Z$, lower neutralized Brin-Katok's local entropy of $\mu$, and neutralized Katok's entropy of $\mu$, respectively. In particular, this extends the main result of \cite{YRCEZX} for autonomous dynamical systems.
\end{abstract}

\maketitle
\thispagestyle{empty}

\section{Introduction and statement of main results}
Throughout this paper $X$ is a compact metric space with metric $d$ and $f_{1,\infty}=(f_n)_{n=1}^\infty$ is a sequence of continuous maps $f_n: X\to X$, such a pair $(X, f_{1,\infty})$ is a non-autonomous dynamical system so-called time-dependent (or NDS for short). By $\mathcal{M}(X)$ we denote the set of all Borel probability measures on $X$.

The non-autonomous systems yield very flexible models than autonomous cases for the study and description of real-world processes. They may be used to describe the evolution of a wider class of phenomena, including systems that are forced or driven. Recently, there have been major efforts in establishing a general theory of such systems (see \cite{BV,JNSFHG1,K1,K2,K3,KR,KS,JNSFHG,JNS000,JNS0000,O,DTRD}), but a global theory is still out of reach. 

In the theory of dynamical systems, entropies are nonnegative extended real numbers measuring the complexity of a dynamical system. Topological entropy was first introduced by Adler et al. \cite{AKM} via open covers for continuous maps in compact topological spaces. In 1970, Bowen \cite{RB} gave another definition in terms of separated and spanning sets for uniformly continuous maps in metric spaces, and this definition is equivalent to Adler's definition for continuous maps in compact metric spaces. Later, for noncompact sets, Bowen \cite{RB4} gave a characterization of dimension type for the entropy which was further investigated by Pesin and Pitskel \cite{PYBPBS}. Also, the measure-theoretic (metric) entropy for an invariant measure was introduced by Kolmogorov \cite{KAAAA}. The basic relation between topological entropy and measure-theoretic entropy is the variational principle, see \cite{PW}. Topological entropy which is an important tool to understand the complexity of dynamical systems plays a vital role in topological dynamics, ergodic theory, mean dimension theory, and other fields of dynamical systems. Moreover, it has close relationships with many important dynamical properties, such as chaos, Lyapunov exponents, the growth of the number of periodic points, and so on. Our main goal in this paper is to focus on the interplay between ergodic theory and topological entropy of non-autonomous dynamical systems.

In 1996, Kolyada and Snoha extended the concept of topological entropy to non-autonomous systems, based on open covers, separated sets, and spanning sets, and obtained a series of important properties of these systems \cite{KS}. Also, Kawan \cite{K1,K2,K3} introduced and studied the notion of metric entropy for
non-autonomous dynamical systems and showed that it is related via variational inequality (principle) to the topological entropy as defined by Kolyada and Snoha. More precisely, for equicontinuous topological
non-autonomous dynamical systems
$(X_{1,\infty},f_{1,\infty})$, he proved the variational inequality
\begin{equation*}
\sup_{\mu_{1,\infty}}h_{\mathcal{E}_{M}}(f_{1,\infty},\mu_{1,\infty})\leq h_{\text{top}}(f_{1,\infty}),
\end{equation*}
where $h_{\text{top}}(f_{1,\infty})$ and $h_{\mathcal{E}_{M}}(f_{1,\infty},\mu_{1,\infty})$ are the topological
and metric entropy of $(X_{1,\infty},f_{1,\infty})$, the supremum is taken over all invariant measure sequences $\mu_{1,\infty}$ and $\mathcal{E}_{M}$ is Misiurewicz class of partitions. Also, he proved the following variational principle
\begin{equation*}
\sup_{\mu_{1,\infty}}h_{\mathcal{E}_{M}}(f_{1,\infty},\mu_{1,\infty})=h_{\text{top}}(f_{1,\infty})
\end{equation*}
for non-autonomous dynamical systems $(M,f_{1,\infty})$ built from $C^{1}$ expanding maps $f_{n}$ on a compact Riemannian manifold $M$ with uniform bounds on expansion factors and derivatives that act in the same way on the fundamental group of $M$. Moreover, Nazarian Sarkooh \cite{JNS0000} showed the following variational principle 
\begin{equation*} 
h_{top}^{B}(f_{1,\infty},Z)=h_{top}^{WB}(f_{1,\infty},Z)=\sup\{\underline{h}_{\mu}(f_{1,\infty}):\mu\in\mathcal{M}(X), \mu(Z)=1\}
\end{equation*}
for non-autonomous dynamical system $(X, f_{1,\infty})$ on a compact metric space $X$ which links the Bowen topological entropy $h_{top}^{B}(f_{1,\infty},Z)$ and weighted Bowen topological entropy $h_{top}^{WB}(f_{1,\infty},Z)$ of non-empty compact subset $Z$ to the measure-theoretic lower entropy  $\underline{h}_{\mu}(f_{1,\infty})$ of Borel probability measures.

Recently, Ovadia and Rodriguez-Hertz \cite{SBOFRH} defined the neutralized Bowen open ball for an autonomous dynamical system $(X,f)$ on a compact metric space $X$ with metric $d$ as
\begin{equation*}
B_{n}(x,\text{e}^{-n\epsilon})=\{y\in X: d(f^{j}(x),f^{j}(y))<\text{e}^{-n\epsilon}\ \text{for}\ 0\leq j<n\},
\end{equation*}
where $\epsilon>0$, $n\in\mathbb{N}$, and $x\in X$. As the Brin-Katok's entropy formula shows, the usual Bowen open balls $\{B_{n}(x,\epsilon)\}$ allow us to control their measure and image under the dynamics for iterations, while it may possess complicated geometric shape in a central direction. Ovadia and Rodriguez-Hertz by replacing the usual Bowen open balls $\{B_{n}(x,\epsilon)\}$ with neutralized Bowen open balls 
$\{B_{n}(x,\text{e}^{-n\epsilon})\}$, defined the lower neutralized Brin-Katok's local entropy to estimate the asymptotic measure of sets with a distinctive geometric shape by neutralizing the subexponential effects. Moreover, the neutralized Bowen open balls have more advantages for describing the neighborhood with a local linearization of the dynamics.

Motivated by \cite{RB4,DFWH,JNS000,SBOFRH}, a natural question is whether the previous variational principles for Bowen topological entropy and weighted Bowen topological entropy on non-empty compact subsets for a non-autonomous dynamical system still holds by considering neutralized Bowen open balls. Hence, we introduce the notions of neutralized Bowen topological entropy of subsets, neutralized weighted Bowen topological entropy of subsets, lower neutralized Brin-Katok's local entropy of Borel probability measures, and neutralized Katok's entropy of Borel probability measures through neutralized Bowen open balls for non-autonomous dynamical systems. Then, by using some ideas from geometric measure theory and utilizing the dimensional approach, we establish the following variational principles for neutralized Bowen topological entropy and neutralized weighted Bowen topological entropy of non-empty compact subsets. In particular, different from the classical variational principles for topological entropy in terms of Kolmogorov-Sinai entropy \cite{PW} and Bowen topological entropy of non-empty compact subsets in terms of lower Brin-Katok's local entropy \cite{DFWH}, the form of our variational principles for neutralized Bowen topological entropy is more close to Lindenstrauss-Tsukamoto's variational principle \cite{LETM} for metric mean dimension in terms of rate-distortion functions. 

\begin{theorem A}
Let $(X,f_{1,\infty})$ be a non-autonomous dynamical system on a compact metric space $(X,d)$ and $Z$ be a non-empty compact subset of $X$. Then
\begin{eqnarray*}
h_{top}^{NB}(f_{1,\infty},Z)=h_{top}^{NWB}(f_{1,\infty},Z)
&=& \lim_{\epsilon\to 0}\sup\{\underline{h}_{\mu}^{NBK}(f_{1,\infty},\epsilon):\mu\in\mathcal{M}(X), \mu(Z)=1\}\\
&=& \lim_{\epsilon\to 0}\sup\{h_{\mu}^{NK}(f_{1,\infty},\epsilon):\mu\in\mathcal{M}(X), \mu(Z)=1\},
\end{eqnarray*}
where $h_{top}^{NB}(f_{1,\infty},Z)$, $h_{top}^{NWB}(f_{1,\infty},Z)$, $\underline{h}_{\mu}^{NBK}(f_{1,\infty},\epsilon)$, and $h_{\mu}^{NK}(f_{1,\infty},\epsilon)$ denote the neutralized Bowen topological entropy of $Z$, neutralized weighted Bowen topological entropy of $Z$, lower neutralized Brin-Katok's local entropy of $\mu$, and neutralized Katok's entropy of $\mu$, respectively.
\end{theorem A}
As a direct consequence of Theorem $A$, we have the following corollary that is the main result of \cite{YRCEZX} for autonomous dynamical systems.
\begin{corollary B}
Let $(X,f)$ be an autonomous dynamical system on a compact metric space $(X,d)$ and $Z$ be a non-empty compact subset of $X$. Then
\begin{eqnarray*}
h_{top}^{NB}(f,Z)=h_{top}^{NWB}(f,Z)
&=& \lim_{\epsilon\to 0}\sup\{\underline{h}_{\mu}^{NBK}(f,\epsilon):\mu\in\mathcal{M}(X), \mu(Z)=1\}\\
&=& \lim_{\epsilon\to 0}\sup\{h_{\mu}^{NK}(f,\epsilon):\mu\in\mathcal{M}(X), \mu(Z)=1\},
\end{eqnarray*}
where $h_{top}^{NB}(f,Z)$, $h_{top}^{NWB}(f,Z)$, $\underline{h}_{\mu}^{NBK}(f,\epsilon)$, and $h_{\mu}^{NK}(f,\epsilon)$ denote the neutralized Bowen topological entropy of $Z$, neutralized weighted Bowen topological entropy of $Z$, lower neutralized Brin-Katok's local entropy of $\mu$, and neutralized Katok's entropy of $\mu$, respectively.
\end{corollary B}
 
\textbf{This paper is organized as follows.} In Section \ref{section2}, we give a precise definition of a
non-autonomous dynamical system and introduce the notions of neutralized Bowen topological entropy of subsets, neutralized weighted Bowen topological entropy of subsets, lower neutralized Brin-Katok's local entropy of Borel probability measures, and neutralized Katok's entropy of Borel probability measures for 
non-autonomous dynamical systems. Moreover, we give a few preparatory results needed for the proof of the main results. The proof of Theorem A is given in Section \ref{section3}.
\section{Definitions and Preliminary}\label{section2}
A \emph{non-autonomous} or \emph{time-dependent} dynamical system (an \emph{NDS} for short), is a pair $(X_{1,\infty}, f_{1,\infty})$, where $X_{1,\infty}=(X_n)_{n=1}^\infty$ is a sequence of sets
and $f_{1,\infty}=(f_n)_{n=1}^\infty$ is a sequence of maps $f_n: X_n \to X_{n+1}$.
If all the sets $X_n$ are compact metric spaces and all the $f_n$ are continuous,
we say that $(X_{1,\infty}, f_{1,\infty})$ is a \emph{topological NDS}.
Here, we assume that $X$ is a compact metric space with metric $d$, all the sets $X_{n}$ are equal to the set $X$ and we abbreviate $(X_{1,\infty}, f_{1,\infty})$ by $(X, f_{1,\infty})$. Throughout this paper, we work with topological NDSs and use NDS instead of topological NDS for simplicity.
The time evolution of the system is defined by composing the maps $f_{n}$ in an obvious way.
In general, we define
\begin{equation*}
f_i^n:=f_{i+n-1}\circ\cdots\circ f_{i+1}\circ f_i \ \ \text{for} \ i,n\in \mathbb{N},\ \text{and} \ f_i^0:=\text{id}_X.
\end{equation*}
Subset $Z$ of $X$ is said to be \emph{forward} $f_{1,\infty}$-\emph{invariant} if  $f_{n}(Z)\subset Z$ for all $n\in\mathbb{N}$.

In what follows, we introduce the notions of neutralized Bowen topological entropy of subsets, neutralized weighted Bowen topological entropy of subsets, lower neutralized Brin-Katok's local entropy of Borel probability measures, and neutralized Katok's entropy of Borel probability measures for NDSs. Moreover, we give a few preparatory results needed for the proof of the main results.
\subsection{Neutralized Bowen topological entropy}
Let $(X, f_{1,\infty})$ be an NDS on a compact metric space $(X,d)$. Given $i,n\in\mathbb{N}$, $x,y\in X$, the $(i,n)$-\emph{Bowen metric} $d_{i,n}$ on $X$ is defined as 
\begin{equation*}
d_{i,n}(x,y):=\max_{0\leq j\leq n-1}d(f_{i}^{j}(x),f_{i}^{j}(y)).
\end{equation*}
Then the \emph{Bowen open ball} of radius $\epsilon>0$ and order $n$ with initial time $i$ around $x$ is given by
\begin{equation*}
B(x;i,n,\epsilon):=\{y \in X: d_{i,n}(x,y)<\epsilon\}.
\end{equation*}

Following the idea of \cite{SBOFRH}, we define the neutralized Bowen open ball for NDSs by replacing the radius $\epsilon$ in the usual Bowen open ball with $\text{e}^{-n\epsilon}$. Precisely, the \emph{neutralized Bowen open ball} of radius $\epsilon>0$ and order $n$ with initial time $i$ around $x$ is given by
\begin{equation*}
B(x;i,n,\text{e}^{-n\epsilon}):=\{y \in X: d_{i,n}(x,y)<\text{e}^{-n\epsilon}\}.
\end{equation*}
For simplicity, denote $B(x;1,n,\epsilon)$, $B(x;1,n,\text{e}^{-n\epsilon})$, and $d_{1,n}$ by $B_{n}(x,\epsilon)$, $B_{n}(x,\text{e}^{-n\epsilon})$, and $d_{n}$, respectively.

Now, by following the idea of \cite{BMKA,DFWH,JNS0000}, we define the so-called neutralized Bowen topological entropy of a subset that is not necessarily compact or forward $f_{1,\infty}$-invariant.

Let $Z$ be a non-empty subset of $X$. Given $n\in\mathbb{N}$, $\alpha\in\mathbb{R}$, and $\epsilon>0$, define
\begin{equation*}
M_{f_{1,\infty}}(n,\alpha,\epsilon,Z):=\inf\bigg\{\sum_{i\in I}\text{e}^{-\alpha n_{i}}\bigg\},
\end{equation*}
where the infimum is taken over all finite or countable collections $\{B_{n_{i}}(x_{i},\text{e}^{-n_{i}\epsilon})\}_{i\in I}$ such that
$x_{i}\in X$, $n_{i}\geq n$, and $Z\subseteq\bigcup_{i\in I}B_{n_{i}}(x_{i},\text{e}^{-n_{i}\epsilon})$. 
The quantity $M_{f_{1,\infty}}(n,\alpha,\epsilon,Z)$ does not decrease as $n$ increases, hence the following limit exists
\begin{equation*}
M_{f_{1,\infty}}(\alpha,\epsilon,Z):=\lim_{n\to\infty}M_{f_{1,\infty}}(n,\alpha,\epsilon,Z).
\end{equation*}
By the construction of Carath\'{e}odory dimension characteristics (see \cite{PYB}), when $\alpha$ goes from
$-\infty$ to $\infty$, the quantity $M_{f_{1,\infty}}(\alpha,\epsilon,Z)$ jumps from $\infty$ to $0$ at a unique critical value. 
Hence we can define the number
\begin{equation*}
M_{f_{1,\infty}}(\epsilon,Z):=\sup\{\alpha:M_{f_{1,\infty}}(\alpha,\epsilon,Z)=\infty\}=\inf\{\alpha:M_{f_{1,\infty}}(\alpha,\epsilon,Z)=0\}.
\end{equation*}
\begin{definition}[Neutralized Bowen topological entropy]
Let $(X, f_{1,\infty})$ be an NDS on a compact metric space $(X,d)$ and $Z$ be a non-empty subset of $X$. Then, the \emph{neutralized Bowen topological entropy} of NDS $(X,f_{1,\infty})$ on the set $Z$ is defined as
\begin{equation*}
h_{top}^{NB}(f_{1,\infty},Z):=\lim_{\epsilon\to 0} M_{f_{1,\infty}}(\epsilon,Z).
\end{equation*}
\end{definition}
\begin{remark}
Denote by $h_{top}^{B}(f_{1,\infty},Z)$ the classical Bowen topological entropy of NDS $(X, f_{1,\infty})$ on a non-empty subset $Z$ of $X$ defined by Bowen open balls $\{B_{n_{i}}(x_{i},\epsilon)\}_{i\in I}$ (see \cite{JNS0000}). Then, it is easy to see that 
$h_{top}^{B}(f_{1,\infty},Z)\leq h_{top}^{NB}(f_{1,\infty},Z)$.
\end{remark}
The following proposition presents some basic properties related to neutralized Bowen topological entropy of NDSs.
\begin{proposition}
Let $(X, f_{1,\infty})$ be an NDS on a compact metric space $(X,d)$ and $Z$, $Z_{1}$, $Z_{2}$ be non-empty subsets of $X$. Then, the following statements hold.
\begin{itemize}
\item[1)]
The value of $h_{top}^{NB}(f_{1,\infty},Z)$ is independent of the choice of metrics on $X$.
\item[2)]
If $Z_{1}\subset Z_{2}$, then $h_{top}^{NB}(f_{1,\infty},Z_{1})\leq h_{top}^{NB}(f_{1,\infty},Z_{2})$.
\item[3)]
If $Z$ is a countable union of $Z_{i}$, then $M_{f_{1,\infty}}(\epsilon,Z)=\sup_{i}M_{f_{1,\infty}}(\epsilon,Z_{i})$ for every $\epsilon>0$.
\item[4)]
If $Z$ is a finite union of $Z_{i}$, then $h_{top}^{NB}(f_{1,\infty},Z)=\max_{i}h_{top}^{NB}(f_{1,\infty},Z_{i})$.
\end{itemize}
\end{proposition}
\begin{proof}
We only show the first statement, because the other statements are a direct consequence of the definition. Let $d_{1}$, $d_{2}$ be two compatible metrics on $X$. Then for every $\epsilon^{\prime}>0$ there is $\delta^{\prime}>0$ such that for all $x,y\in X$, $d_{1}(x,y)<\delta^{\prime}$ implies $d_{2}(x,y)<\epsilon^{\prime}$. Now, fix $\epsilon>0$ and let $0<\delta<\epsilon$. For every $n\geq 1$, there is $n_{0}\geq n$ (depends only on $\epsilon$, $\delta$, and $n$) so that for all $x,y\in X$, $d_{1}(x,y)<\text{e}^{-n_{0}\delta}$ implies $d_{2}(x,y)<\text{e}^{-n\epsilon}$. Hence $M_{f_{1,\infty}}^{d_{2}}(n,\alpha,\epsilon,Z)\leq M_{f_{1,\infty}}^{d_{1}}(n_{0},\alpha,\delta,Z)\leq M_{f_{1,\infty}}^{d_{1}}(\alpha,\delta,Z)$ and so $M_{f_{1,\infty}}^{d_{2}}(\alpha,\epsilon,Z)\leq M_{f_{1,\infty}}^{d_{1}}(\alpha,\delta,Z)$. Consequently, $M_{f_{1,\infty}}^{d_{2}}(\epsilon,Z)\leq M_{f_{1,\infty}}^{d_{1}}(\delta,Z)$. As $\epsilon\to 0$, one has $h_{top}^{NB}(f_{1,\infty},d_{2},Z)\leq h_{top}^{NB}(f_{1,\infty},d_{1},Z)$. Now, by exchanging the role of $d_{1}$ and $d_{2}$ we get the converse inequality. 
\end{proof}
\subsection{Neutralized weighted Bowen topological entropy}
Let $(X, f_{1,\infty})$ be an NDS on a compact metric space $(X,d)$ and $Z$ be a non-empty subset of $X$. For any $n\in\mathbb{N}$, $\alpha\in\mathbb{R}$, $\epsilon>0$, and any bounded function $g:X\to\mathbb{R}$, define
\begin{equation*}
\mathcal{W}_{f_{1,\infty}}(n,\alpha,\epsilon,g):=\inf\bigg\{\sum_{i\in I}c_{i}\text{e}^{-\alpha n_{i}}\bigg\},
\end{equation*}
where the infimum is taken over all finite or countable families
$\{(B_{n_{i}}(x_{i},\text{e}^{-n_{i}\epsilon}),c_{i})\}_{i\in I}$ such that
$0<c_{i}<\infty$, $x_{i}\in X$, $n_{i}\geq n$ for all $i$, and $\sum_{i\in I}c_{i}\chi_{B_{n_{i}}(x_{i},\text{e}^{-n_{i}\epsilon})}\geq g$,
where $\chi_{B}$ denotes the characteristic function of subset $B$ of $X$. Now, set
\begin{equation*}
\mathcal{W}_{f_{1,\infty}}(n,\alpha,\epsilon,Z):=\mathcal{W}_{f_{1,\infty}}(n,\alpha,\epsilon,\chi_{Z}).
\end{equation*}
The quantity $\mathcal{W}_{f_{1,\infty}}(n,\alpha,\epsilon,Z)$ does not decrease as $n$ increases, hence the following limit exists
\begin{equation*}
\mathcal{W}_{f_{1,\infty}}(\alpha,\epsilon,Z):=\lim_{n\to\infty}\mathcal{W}_{f_{1,\infty}}(n,\alpha,\epsilon,Z).
\end{equation*}
By the construction of Carath\'{e}odory dimension characteristics (see \cite{PYB}), when $\alpha$ goes from
$-\infty$ to $\infty$, the quantity $\mathcal{W}_{f_{1,\infty}}(\alpha,\epsilon,Z)$ jumps from $\infty$ to $0$ at a unique critical value. Hence we can define the number
\begin{equation*} 
\mathcal{W}_{f_{1,\infty}}(\epsilon,Z):=\sup\{\alpha:\mathcal{W}_{f_{1,\infty}}(\alpha,\epsilon,Z)=\infty\}=\inf\{\alpha:\mathcal{W}_{f_{1,\infty}}(\alpha,\epsilon,Z)=0\}.
\end{equation*}
\begin{definition}[Neutralized weighted Bowen topological entropy]
Let $(X, f_{1,\infty})$ be an NDS on a compact metric space $(X,d)$ and $Z$ be a non-empty subset of $X$. Then, the \emph{neutralized weighted Bowen topological entropy} of NDS $(X,f_{1,\infty})$ on the set $Z$ is defined as
\begin{equation*} 
h_{top}^{NWB}(f_{1,\infty},Z):=\lim_{\epsilon\to 0}\mathcal{W}_{f_{1,\infty}}(\epsilon,Z).
\end{equation*}
\end{definition}
In what follows, as the role of Kolmogorov-Sinai entropy played in the classical variational principle of topological entropy of autonomous dynamical systems (see \cite{PW}), we define the lower neutralized Brin-Katok's local entropy and the neutralized Katok's entropy of Borel probability measures for NDSs, because we need them to establish the variational principle for neutralized Bowen topological entropy and neutralized weighted Bowen topological entropy on non-empty compact subsets of NDSs.
\subsection{Lower neutralized Brin-Katok's local entropy}
Let $(X, f_{1,\infty})$ be an NDS on a compact metric space $(X,d)$, and let $\mathcal{M}(X)$ denote the set of Borel probability measures on $X$. By following the idea of \cite{BMKA,DFWH,JNS0000}, one can employ the neutralized Bowen open balls to define the local lower neutralized Brin-Katok's entropy for NDSs. For $\mu\in\mathcal{M}(X)$, $\epsilon>0$, and $x\in X$, we define
\begin{equation*}
\underline{h}_{\mu}^{NBK}(f_{1,\infty},\epsilon,x):=\liminf_{n\to\infty}\dfrac{-\log\mu(B_{n}(x,\text{e}^{-n\epsilon}))}{n}
\end{equation*}
and
\begin{equation*}
\underline{h}_{\mu}^{NBK}(f_{1,\infty},\epsilon):=\int\underline{h}_{\mu}^{NBK}(f_{1,\infty},\epsilon,x) d\mu.
\end{equation*}
Then the \emph{lower neutralized Brin-Katok's local entropy} of $\mu$ is defined as
\begin{equation*}
\underline{h}_{\mu}^{NBK}(f_{1,\infty}):=\lim_{\epsilon\to 0}\underline{h}_{\mu}^{NBK}(f_{1,\infty},\epsilon).
\end{equation*}
\begin{remark}
Note that $\underline{h}_{\mu}^{NBK}(f_{1,\infty},\epsilon,x)$ is integrable and so $\underline{h}_{\mu}^{NBK}(f_{1,\infty},\epsilon)$ is well-defined. Indeed, for every $\alpha>0$ the set $\{x\in X: \mu(B_{n}(x,\text{e}^{-n\epsilon}))>\alpha\}$ is open, that implies the mappings $x\in X\mapsto\mu(B_{n}(x,\text{e}^{-n\epsilon}))$ and $x\in X\mapsto\underline{h}_{\mu}^{NBK}(f_{1,\infty},\epsilon,x)$ are measurable.
\end{remark}
\subsection{Neutralized Katok's entropy.}
Let $(X, f_{1,\infty})$ be an NDS on a compact metric space $(X,d)$. Then, the neutralized Katok's entropy of $(X, f_{1,\infty})$ can be defined similarly to the classical Katok's entropy defined by spanning sets and separated sets \cite{KA}. However, to pursue a variational principle for $(X, f_{1,\infty})$ we need to define neutralized Katok's entropy of Borel probability measures by Carath\'{e}odory Pesin structure \cite{PYB}.

Let $n\in\mathbb{N}$, $\epsilon>0$, $\alpha\in\mathbb{R}$, $\mu\in\mathcal{M}(X)$, and $0<\delta<1$. Put 
\begin{equation*}
\Lambda_{f_{1,\infty}}(n,\alpha,\epsilon,\mu,\delta):=\inf\bigg\{\sum_{i\in I}\text{e}^{-\alpha n_{i}}\bigg\},
\end{equation*}
where the infimum is taken over all finite or countable collections $\{B_{n_{i}}(x_{i},\text{e}^{-n_{i}\epsilon})\}_{i\in I}$ such that $x_{i}\in X$, $n_{i}\geq n$, and $\mu(\bigcup_{i\in I}B_{n_{i}}(x_{i},\text{e}^{-n_{i}\epsilon}))>1-\delta$. The quantity $\Lambda_{f_{1,\infty}}(n,\alpha,\epsilon,\mu,\delta)$ does not
decrease as $n$ increases, hence the following limit exists
\begin{equation*}
\Lambda_{f_{1,\infty}}(\alpha,\epsilon,\mu,\delta):=\lim_{n\to\infty}\Lambda_{f_{1,\infty}}(n,\alpha,\epsilon,\mu,\delta).
\end{equation*}
By the construction of Carath\'{e}odory dimension characteristics (see \cite{PYB}), when $\alpha$ goes from
$-\infty$ to $\infty$, the quantity $\Lambda_{f_{1,\infty}}(\alpha,\epsilon,\mu,\delta)$ jumps from $\infty$ to $0$ at a unique critical value. Hence we can define the number
\begin{equation*}
\Lambda_{f_{1,\infty}}(\epsilon,\mu,\delta):=\sup\{\alpha:\Lambda_{f_{1,\infty}}(\alpha,\epsilon,\mu,\delta)=\infty\}=\inf\{\alpha:\Lambda_{f_{1,\infty}}(\alpha,\epsilon,\mu,\delta)=0\}.
\end{equation*}
Put $h_{\mu}^{NK}(f_{1,\infty},\epsilon):=\lim_{\delta\to 0}\Lambda_{f_{1,\infty}}(\epsilon,\mu,\delta)$. Then,
the \emph{neutralized Katok's entropy} of $\mu$ is defined as
\begin{equation*}
h_{\mu}^{NK}(f_{1,\infty}):=\lim_{\epsilon\to 0}h_{\mu}^{NK}(f_{1,\infty},\epsilon).
\end{equation*}

The following proposition presents some basic properties related to lower neutralized Brin-Katok's local entropy and the neutralized Katok's entropy of Borel probability measures for NDSs.
\begin{proposition}\label{propjj}
Let $(X, f_{1,\infty})$ be an NDS on a compact metric space $(X,d)$ and $\mu\in\mathcal{M}(X)$.
Then, for every $\epsilon>0$, one has $\underline{h}_{\mu}^{NBK}(f_{1,\infty},\dfrac{\epsilon}{2})\leq h_{\mu}^{NK}(f_{1,\infty},\epsilon)$.
\end{proposition}
\begin{proof}
Let $\underline{h}_{\mu}^{NBK}(f_{1,\infty},\dfrac{\epsilon}{2})>\alpha>0$. For each $n\in\mathbb{N}$, set
\begin{equation*}
E_{n}:=\big\{x\in X: \mu(B_{m}(x,\text{e}^{\frac{-m\epsilon}{2}}))<\text{e}^{-m\alpha}\ \text{for all}\ m\geq n\big\}.
\end{equation*}
Then, there is $n_{0}\in\mathbb{N}$ such that $\text{e}^{\frac{n_{0}\epsilon}{2}}>2$ and $\mu(E_{n_{0}})>0$. Put $\delta_{0}:=\frac{1}{2}\mu(E_{n_{0}})$. Let $\{B_{n_{i}}(x_{i},\text{e}^{-n_{i}\epsilon})\}_{i\in I}$ be a finite or countable collection such that $x_{i}\in X$, $n_{i}\geq n_{0}$, and $\mu(\bigcup_{i\in I}B_{n_{i}}(x_{i},\text{e}^{-n_{i}\epsilon}))>1-\delta_{0}$. Then $\mu(E_{n_{0}}\cap\bigcup_{i\in I}B_{n_{i}}(x_{i},\text{e}^{-n_{i}\epsilon}))\geq\delta_{0}$. Put $I_{1}:=\{i\in I: E_{n_{0}}\cap B_{n_{i}}(x_{i},\text{e}^{-n_{i}\epsilon})\neq\emptyset\}$. For every $i\in I_{1}$, take $y_{i}\in E_{n_{0}}\cap B_{n_{i}}(x_{i},\text{e}^{-n_{i}\epsilon})$ such that
\begin{equation*}
E_{n_{0}}\cap B_{n_{i}}(x_{i},\text{e}^{-n_{i}\epsilon})\subset B_{n_{i}}(y_{i},2\text{e}^{-n_{i}\epsilon})\subset B_{n_{i}}(y_{i},\text{e}^{\frac{-n_{i}\epsilon}{2}}).
\end{equation*}
Thus,
\begin{equation*}
\sum_{i\in I}\text{e}^{-\alpha n_{i}}\geq\sum_{i\in I_{1}}\text{e}^{-\alpha n_{i}}\geq\sum_{i\in I_{1}}\mu(B_{n_{i}}(y_{i},\text{e}^{\frac{-n_{i}\epsilon}{2}}))\geq\delta_{0},
\end{equation*}
that implies $\Lambda_{f_{1,\infty}}(\alpha,\epsilon,\mu,\delta_{0})\geq\Lambda_{f_{1,\infty}}(n_{0},\alpha,\epsilon,\mu,\delta_{0})$ and hence $\Lambda_{f_{1,\infty}}(\epsilon,\mu,\delta_{0})\geq\alpha$. Consequently, $h_{\mu}^{NK}(f_{1,\infty},\epsilon)\geq\alpha$. This finishes the proof as $\alpha\to \underline{h}_{\mu}^{NBK}(f_{1,\infty},\dfrac{\epsilon}{2})$.
\end{proof}
\section{Proof of main results}\label{section3}
In this section, we prove Theorem A. To do so, we use the idea of geometric measure theory \cite{MP} and the works of Feng and Huang \cite{DFWH} and Nazarian Sarkooh \cite{JNS0000} to show that the neutralized Bowen topological entropy and the neutralized weighted Bowen topological entropy of NDSs are equivalent. This allows us to define a positive linear functional $L$ on $\mathcal{C}(X,\mathbb{R})$ by applying Frostman's lemma which is an analog of Feng and Huang's approximation, where $\mathcal{C}(X,\mathbb{R})$ denotes the Banach space of all continuous real-valued functions on $X$ equipped with the supremum norm. Then \emph{Riesz representation theorem} can be applied to produce a Borel probability measure $\mu\in\mathcal{M}(X)$ with $\mu(Z)=1$ so that $M_{f_{1,\infty}}(\epsilon,Z)\leq\underline{h}_{\mu}^{NBK}(f_{1,\infty},2\epsilon)$. 

In Theorem \ref{propj}, we prove that the neutralized Bowen topological entropy is equal to the neutralized weighted Bowen topological entropy, i.e., $h_{top}^{NB}(f_{1,\infty},Z)=h_{top}^{NWB}(f_{1,\infty},Z)$. To do this, we need the following lemma \cite[Lemma 6.3]{TW} which is called the \emph{Vitali covering lemma} (see \cite[Theorem 2.1]{MP}).
\begin{lemma}\label{lemmaj}
Let $(X,d)$ be a compact metric space, $r>0$, and $\{B(x_{i},r)\}_{i\in I}$ be a family of open balls in $X$. Set
$I(i):=\{j\in I: B(x_{i},r)\cap B(x_{j},r)\neq\emptyset\}$. Then there exists a finite subset $J\subset I$ so that for any $i,j\in J$ with $i\neq j$, $I(i)\cap I(j)\neq\emptyset$ and $\bigcup_{i\in I}B(x_{i},r)\subseteq\bigcup_{j\in J}B(x_{j},5r)$.
\end{lemma}
\begin{theorem}\label{propj}
Let $(X, f_{1,\infty})$ be an NDS on a compact metric space $(X,d)$ and $Z$ be a non-empty subset of $X$. Then,  for $\epsilon>0$, $\alpha\in\mathbb{R}$, $\theta>0$, and sufficiently large $n$, we have
$M_{f_{1,\infty}}(n,\alpha+\theta,\frac{\epsilon}{2},Z)\leq\mathcal{W}_{f_{1,\infty}}(n,\alpha,\epsilon,Z)\leq M_{f_{1,\infty}}(n,\alpha,\epsilon,Z)$. Consequently, $h_{top}^{NB}(f_{1,\infty},Z)=h_{top}^{NWB}(f_{1,\infty},Z)$.
\end{theorem}
\begin{proof}
The inequality $\mathcal{W}_{f_{1,\infty}}(n,\alpha,\epsilon,Z)\leq M_{f_{1,\infty}}(n,\alpha,\epsilon,Z)$ follows by definitions. We show $M_{f_{1,\infty}}(n,\alpha+\theta,\frac{\epsilon}{2},Z)\leq\mathcal{W}_{f_{1,\infty}}(n,\alpha,\epsilon,Z)$ by modifying the proof of \cite[Proposition 6.4]{TW}.

Let $n$ be a sufficiently large integer so that $\text{e}^{\frac{m\epsilon}{2}}>5$ and $\frac{m^{2}}{\text{e}^{m\theta}}<1$ for all $m\geq n$. Let $\{(B_{n_{i}}(x_{i},\text{e}^{-n_{i}\epsilon}),c_{i})\}_{i\in I}$ with $0<c_{i}<\infty$, $x_{i}\in X$, and $n_{i}\geq n$ for all $i$, be a finite or countable family satisfying
$\sum_{i\in I}c_{i}\chi_{B_{n_{i}}(x_{i},\text{e}^{-n_{i}\epsilon})}\geq\chi_{Z}$. For $m\in\mathbb{N}$, put
$I_{m}:=\{i\in I: n_{i}=m\}$. Let $t>0$ and $m\geq n$. Set
\begin{equation*}
Z_{m,t}:=\{x\in Z: \sum_{i\in I_{m}}c_{i}\chi_{B_{m}(x_{i},\text{e}^{-m\epsilon})}(x)>t\}
\end{equation*}
and
\begin{equation*}
I_{m}^{t}:=\{i\in I_{m}: B_{m}(x_{i},\text{e}^{-m\epsilon})\cap Z_{m,t}\neq\emptyset\}.
\end{equation*}
Then, $Z_{m,t}\subset\bigcup_{i\in I_{m}^{t}}B_{m}(x_{i},\text{e}^{-m\epsilon})$. Let $\mathcal{B}:=\{B_{m}(x_{i},\text{e}^{-m\epsilon})\}_{i\in I_{m}^{t}}$. By Lemma \ref{lemmaj}, there is a finite subset $J\subset I_{m}^{t}$ such that
\begin{equation*}
\bigcup_{i\in I_{m}^{t}}B_{m}(x_{i},\text{e}^{-m\epsilon})\subset\bigcup_{j\in J}B_{m}(x_{j},5\text{e}^{-m\epsilon})\subset\bigcup_{j\in J}B_{m}(x_{j},\text{e}^{\frac{-m\epsilon}{2}}),
\end{equation*}
and $I_{m}^{t}(i)\cap I_{m}^{t}(j)\neq\emptyset$ for any $i,j\in J$ with $i\neq j$, where
$I_{m}^{t}(i):=\{j\in I_{m}^{t}: B_{m}(x_{j},\text{e}^{-m\epsilon})\cap B_{m}(x_{i},\text{e}^{-m\epsilon})\neq\emptyset\}$. 

Now, for each $j\in J$, we choose $y_{j}\in B_{m}(x_{j},\text{e}^{-m\epsilon})\cap Z_{m,t}$. Then $\sum_{i\in I_{m}^{t}}c_{i}\chi_{B_{m}(x_{i},\text{e}^{-m\epsilon})}(y_{j})>t$ and hence $\sum_{i\in I_{m}^{t}(j)}c_{i}>t$. Consequently,
\begin{equation*}
|J|<\frac{1}{t}\sum_{j\in J}\sum_{i\in I_{m}^{t}(j)}c_{i}\leq\frac{1}{t}\sum_{i\in I_{m}^{t}}c_{i}.
\end{equation*}
It follows that $M_{f_{1,\infty}}(n,\alpha+\theta,\frac{\epsilon}{2},Z_{m,t})\leq|J|.\text{e}^{-m(\alpha+\theta)}\leq\frac{1}{m^{2}t}\sum_{i\in I_{m}}c_{i}\text{e}^{-m\alpha}$. Note that for any $0<t<1$, $Z=\cup_{m\geq n}Z_{m,\frac{t}{m^{2}}}$, and so  
\begin{equation*}
M_{f_{1,\infty}}(n,\alpha+\theta,\frac{\epsilon}{2},Z)\leq\sum_{m\geq n}M_{f_{1,\infty}}(n,\alpha+\theta,\frac{\epsilon}{2},Z_{m,\frac{t}{m^{2}}})\leq\frac{1}{t}\sum_{i\in I}c_{i}\text{e}^{-n_{i}\alpha}.
\end{equation*}
Hence, as $t\to 1$, we have $M_{f_{1,\infty}}(n,\alpha+\theta,\frac{\epsilon}{2},Z)\leq\sum_{i\in I}c_{i}\text{e}^{-n_{i}\alpha}$. So $M_{f_{1,\infty}}(n,\alpha+\theta,\frac{\epsilon}{2},Z)\leq\mathcal{W}_{f_{1,\infty}}(n,\alpha,\epsilon,Z)$ that implies the desired result.
\end{proof}
To prove Theorem A, we need the following dynamical Frostman's lemma which is an analog of Feng and Huang's approximation. We omit the detail of its proof because the proof of \cite[Lemma 3.4]{DFWH} and \cite[Lemma 4.3]{JNS0000} which are adapted from Howroyd's elegant argument (see \cite[Theorem 2]{HJD} and \cite[Theorem 8.17]{MP}) also works for it if the Bowen open ball $B_{m}(x,\epsilon)$ is replaced with neutralized Bowen open ball $B_{m}(x,\text{e}^{-m\epsilon})$.
\begin{lemma}\label{lemmajjj}
Let $(X, f_{1,\infty})$ be an NDS on a compact metric space $(X,d)$ and $Z$ be a non-empty compact subset of $X$. Let $\alpha\in\mathbb{R}$, $n\in\mathbb{N}$, and $\epsilon>0$. Set $c:=\mathcal{W}_{f_{1,\infty}}(n,\alpha,\epsilon,Z)>0$. Then there is $\mu\in\mathcal{M}(X)$ such that $\mu(Z)=1$ and
\begin{equation*}
\mu(B_{m}(x,\text{e}^{-m\epsilon}))\leq\dfrac{1}{c}\text{e}^{-\alpha m},\ \ \ \forall x\in X,\ m\geq n.
\end{equation*}
\end{lemma}
\textbf{Proof of Theorem A}.
Note that for every $\mu\in\mathcal{M}(X)$ with $\mu(Z)=1$ and $\epsilon>0$, $h_{\mu}^{NK}(f_{1,\infty},\epsilon)\leq M_{f_{1,\infty}}(\epsilon,Z)$. Hence, by Proposition \ref{propjj}, we have
\begin{eqnarray*}
&&
\lim_{\epsilon\to 0}\sup\{\underline{h}_{\mu}^{NBK}(f_{1,\infty},\epsilon): \mu\in\mathcal{M}(X),\mu(Z)=1\}\\
&\leq & \lim_{\epsilon\to 0}\sup\{h_{\mu}^{NK}(f_{1,\infty},\epsilon): \mu\in\mathcal{M}(X),\mu(Z)=1\}\\
&\leq & h_{top}^{NB}(f_{1,\infty},Z).
\end{eqnarray*}
Now, fix $\epsilon>0$ and assume that $M_{f_{1,\infty}}(\epsilon,Z)>\alpha>0$. Then, by Theorem \ref{propj}, one has $\mathcal{W}_{f_{1,\infty}}(\alpha,2\epsilon,Z)>0$. So there is $n\in\mathbb{N}$ such that $c:=\mathcal{W}_{f_{1,\infty}}(n,\alpha,2\epsilon,Z)>0$. By Lemma \ref{lemmajjj}, there is $\mu\in\mathcal{M}(X)$ such that $\mu(Z)=1$ and 
\begin{equation*}
\mu(B_{m}(x,\text{e}^{-2m\epsilon}))\leq\dfrac{1}{c}\text{e}^{-\alpha m},\ \ \ \forall x\in X,\ m\geq n.
\end{equation*}
This implies that $\underline{h}_{\mu}^{NBK}(f_{1,\infty},2\epsilon)\geq\alpha$. Now, as $\alpha\to M_{f_{1,\infty}}(\epsilon,Z)$, we obtain that
\begin{equation*}
M_{f_{1,\infty}}(\epsilon,Z)\leq\underline{h}_{\mu}^{NBK}(f_{1,\infty},2\epsilon)\leq\sup\{\underline{h}_{\mu}^{NBK}(f_{1,\infty},2\epsilon): \mu\in\mathcal{M}(X),\mu(Z)=1\}.
\end{equation*}
The above statement together with Theorem \ref{propj} finishes the proof of Theorem A.
\begin{remark}
Note that, one can not directly exchange the order of $\lim_{\epsilon\to 0}$ and $\sup$ in Theorem A, because of
\begin{equation*}
\underline{h}_{\mu}^{NBK}(f_{1,\infty})=\inf_{\epsilon>0}\underline{h}_{\mu}^{NBK}(f_{1,\infty},\epsilon)\ \text{and}\ h_{\mu}^{NK}(f_{1,\infty})=\inf_{\epsilon>0}h_{\mu}^{NK}(f_{1,\infty},\epsilon).
\end{equation*}
This obstacle coming from definition does not allow us to establish the variational principle for neutralized Bowen topological entropy and neutralized weighted Bowen topological entropy whose form is closer to the Nazarian Sarkooh work \cite[Corollary D]{JNS0000}. Interested readers also can see \cite[Theorem 1.1]{YRCEZX} for analogous problems in autonomous dynamical systems, and \cite[Theorem 1.4]{TW} and \cite[Theorem 1.4]{YRCEZX1} in metric mean dimension theory for an analog problem.
\end{remark}
\section*{Acknowledgments}
The authors would like to thank the respectful referee for his/her comments on the manuscript.
\section*{Declarations} 
\begin{itemize}
\item[]\textbf{Ethical Approval}: Not applicable. 
\item[]\textbf{Competing interests}: There is no conflict of interest.
\item[]\textbf{Authors' contributions}: All the authors were responsible for conceptualizing, writing, reviewing, and editing the manuscript.
\item[]\textbf{Funding}: No Funding.
\item[]\textbf{Availability of data and materials}: Not applicable.
\end{itemize}


\begin{thebibliography}{99}
\bibitem{AKM} Adler, R., Konheim, A., McAndrew, M.: Topological entropy, Trans. Amer. Math. Soc. {\bf 114}, 309--319 (1965)



\bibitem{BV} Barreira, L., Valls, C.: Stability of nonautonomous differential equations, Lecture notes in mathematics, vol. 1926. Springer-Verlag, Berlin Heidelberg (2008)

\bibitem{RB} Bowen, R.: Entropy for group endomorphisms and homogeneous spaces, Trans. Am. Math. Soc. 
{\bf 153} 401--414 (1971)

\bibitem{RB4} Bowen, R.: Topological entropy for noncompact sets, Trans. Am. Math. Soc.
{\bf 184}, 125--136 (1973)




\bibitem{BMKA} Brin, M., Katok, A.: On local entropy, Geometric dynamics (Rio de Janeiro), Lecture Notes in Mathematics, Springer, Berlin {\bf 1007}, 30--38 (1983)

\bibitem{DFWH} Feng, D., Huang, W.: Variational principles for topological entropies of subsets, J. Funct. Anal. {\bf 263}, 2228--2254 (2012)

\bibitem{JNSFHG1} Ghane, F. H., Nazarian Sarkooh, J.: On topological entropy and topological pressure of non-autonomous iterated function systems, J. Korean Math. Soc. {\bf 56}, 1561--1597 (2019)

\bibitem{HJD} Howroyd, J. D.: On dimension and on the existence of sets of finite positive Hausdorff measure, Proc. Lond. Math. Soc. {\bf 70}, 581--604 (1995)




\bibitem{KA} Katok, A.: Lyapunov exponents, entropy and periodic orbits for diffeomorphisms, Publ. Math. Inst. Hautes Etudes Sci. {\bf 51}, 137--173 (1980)

\bibitem{K1} Kawan, C.: Metric entropy of nonautonomous dynamical systems, Nonauton. Stoch. Dyn. Syst. {\bf 1}, 26--52 (2013)

\bibitem{K2} Kawan, C.: Expanding and expansive time-dependent dynamics, Nonlinearity. {\bf 28}, 669--695 (2015)

\bibitem{K3} Kawan, C., Latushkin, Y.: Some results on the entropy of non-autonomous dynamical systems, Dynamical Systems. {\bf 28}, 1--29 (2015)

\bibitem{KR} Kloeden, P. E., Rasmussen, M.: Nonautonomous dynamical systems, Mathematical surveys, and monographs, vol. 176. American Mathematical Society (2011)

\bibitem{KAAAA} Kolmogorov, A. N.: A new metric invariant of transient dynamical systems and automorphisms of Lebesgue spaces, Dokl. Akad. Soc. SSSR {\bf 119}, 861--864 (1958)

\bibitem{KS} Kolyada, S., Snoha, L.: Topological entropy of nonautonomous dynamical systems, Random Comput. Dyn. {\bf 4}, 205--233 (1996)

\bibitem{LETM} Lindenstrauss, E., Tsukamoto, M.: From rate distortion theory to metric mean dimension: variational principle, IEEE Trans. Inform. Theory {\bf 64}, 3590--3609 (2018)


\bibitem{MP} Mattila, P.: Geometry of Sets and Measures in Euclidean Spaces, Cambridge University Press, Cambridge
(1995)



\bibitem{JNSFHG} Nazarian Sarkooh, J., Ghane, F. H.: Specification and thermodynamic properties of topological time-dependent dynamical systems, Qual. Theory Dyn. Syst. {\bf 18}, 1161--1190 (2019)

\bibitem{JNS000} Nazarian Sarkooh, J.: Various shadowing properties for time varying maps, Bull. Korean Math. Soc. {\bf 59}(2), 481--506 (2022)

\bibitem{JNS0000} Nazarian Sarkooh, J.: Variational principles on subsets of non-autonomous dynamical systems: topological pressure and topological entropy, https://arxiv.org/abs/2206.00714


\bibitem{O} Ott, W., Stendlund, M., Young, L. S.: Memory loss for time-dependent dynamical systems, Math. Res. Lett. {\bf 16}, 463--475 (2009)

\bibitem{SBOFRH} Ovadia, S. B., Rodriguez-Hertz, F.: Neutralized local entropy, arXiv:2302.10874.

\bibitem{PYB} Pesin, Y. B.: Dimension Theory in Dynamical Systems: Contemporary Views and Applications, University of Chicago Press, Chicago (1997)

\bibitem{PYBPBS} Pesin, Y. B., Pitskel, B. S.: Topological pressure and the variational principle for noncompact sets,
Funct. Anal. Appl. {\bf 18}(4), 307--318 (1984)






\bibitem{DTRD} Thakkar, D., Das, R.: Topological stability of a sequence of maps on a compact metric space, Bull. Math. Sci. {\bf 4}, 99--111 (2014)


\bibitem{PW} Walters, P.: An introduction to ergodic theory, Springer-Verlag, New York (1982)


\bibitem{TW} Wang, T.: Variational relations for metric mean dimension and rate distortion dimension, Discrete Contin. Dyn. Syst. {\bf 27}, 4593--4608 (2021)

\bibitem{YRCEZX1} YANG, R., CHEN, E., ZHOU, X.: Bowen's equations for upper metric mean dimension with potential, Nonlinearity {\bf 35}, 4905--4938 (2022)

\bibitem{YRCEZX} YANG, R., CHEN, E., ZHOU, X.: Variational principle for neutralized Bowen topological entropy, 
https://arxiv.org/abs/2303.01738v1



\end{thebibliography}
\end{document}